\documentclass[reqno,11pt]{article} 
\usepackage{amsmath, amssymb}
\usepackage{amsthm} 
\newcommand{\Mod}[1]{\ (\textup{mod}\ #1)}
\theoremstyle{plain} 
\newtheorem{theorem}{\indent\sc Theorem}[section]
\newtheorem{lemma}[theorem]{\indent\sc Lemma}
\newtheorem{corollary}[theorem]{\indent\sc Corollary}
\newtheorem{proposition}[theorem]{\indent\sc Proposition}

\theoremstyle{definition} 

\newtheorem{remark}[theorem]{\indent\sc Remark}

\makeatletter
\def\address#1#2{\begingroup
\noindent\parbox[t]{7.8cm}{%
\small{\scshape\ignorespaces#1}\par\vskip1ex
\noindent\small{\itshape E-mail address}%
\/: #2\par\vskip4ex}\hfill%
\endgroup}%
\makeatother
\title{Determination of Fricke families} 
\author{
\textsc{Ick Sun Eum and Dong Hwa Shin$^*$} 
}
\date{} 
\begin{document}

\allowdisplaybreaks

\maketitle

\footnote{
\thanks{$^*$The corresponding author was supported by Hankuk University of Foreign Studies Research Fund of 2015.} }

\begin{abstract}
For a positive integer $N$ divisible by $4$, let
$\mathcal{O}^1_N(\mathbb{Q})$ be
the ring of weakly holomorphic modular functions for the congruence subgroup $\Gamma^1(N)$
with rational Fourier coefficients.
We construct explicit generators of $\mathcal{O}^1_N(\mathbb{Q})$ over $\mathbb{Q}$, from
which we classify all Fricke families of level $N$.
\end{abstract}

\section {Introduction}

For a positive integer $N$, let $\mathcal{F}_N$ be the field of meromorphic modular functions
for the principal congruence subgroup
$\Gamma(N)=\left\{\gamma\in\mathrm{SL}_2(\mathbb{Z})~|~
\gamma\equiv I_2\Mod{N}\right\}$
whose Fourier coefficients belong to the $N$th
cyclotomic field $\mathbb{Q}(\zeta_N)$ where $\zeta_N=e^{2\pi i/N}$. It is well known that $\mathcal{F}_1$ is generated over $\mathbb{Q}$ by
the elliptic modular function $j(\tau)$ ($\tau\in\mathbb{H}$, the complex upper half-plane),
and $\mathcal{F}_N$ is a Galois extension of $\mathcal{F}_1$ with
\begin{equation}\label{Gal}
\mathrm{Gal}(\mathcal{F}_N/\mathcal{F}_1)\simeq\mathrm{GL}_2(\mathbb{Z}/N\mathbb{Z})/\{\pm I_2\}
\end{equation}
(see $\S$2).
For $N\geq2$ we let
\begin{equation*}
\mathcal{V}_N=\{\mathbf{v}\in\mathbb{Q}^2~|~\mathbf{v}~\textrm{has primitive denominator $N$}\}.
\end{equation*}
We call a family $\{h_\mathbf{v}(\tau)\}_{\mathbf{v}\in\mathcal{V}_N}$
of functions in $\mathcal{F}_N$ a \textit{Fricke family} of level $N$ if it satisfies the
following three conditions:
\begin{itemize}
\item[(F1)] Each $h_\mathbf{v}(\tau)$ is weakly holomorphic (that is, holomorphic on $\mathbb{H}$).
\item[(F2)] $h_\mathbf{v}(\tau)$ depends only on $\pm\mathbf{v}\Mod{\mathbb{Z}^2}$.
\item[(F3)] $h_\mathbf{v}(\tau)^\gamma=h_{{^t}\gamma\mathbf{v}}(\tau)$
for all $\gamma\in\mathrm{GL}_2(\mathbb{Z}/N\mathbb{Z})/\{\pm I_2\}$, where
${^t}\gamma$ means the transpose of $\gamma$.
\end{itemize}
There are two important kinds of Fricke families $\{f_\mathbf{v}(\tau)\}_\mathbf{v}$ and
$\{g_\mathbf{v}(\tau)^{12N}\}_\mathbf{v}$, one consisting of Fricke functions and
the other consisting of $12N$th powers of Siegel functions (see $\S$\ref{sectwo}).
They are building blocks of fields of modular functions and groups of modular units
(\cite[Chapter 2]{K-L} and \cite[Chapter 6]{Lang}). Furthermore,
their special values at imaginary quadratic arguments generate class fields over the corresponding imaginary quadratic fields (\cite{J-K-S}, \cite{J-K-S2} and \cite[Chapter 10]{Lang}).
\par
In this paper we shall classify all Fricke families of level $N$ when $N\equiv0\Mod{4}$ (Theorems \ref{isomorphic}, \ref{integral} and Corollary \ref{expression}). Moreover, if we weaken the condition (F1) as
\begin{equation*}
\textrm{(F1$^\prime$) every $h_\mathbf{v}(\tau)$ is holomorphic on $\mathbb{H}$ except for the set}~ \{\gamma(\zeta_3),\gamma(\zeta_4)~|~\gamma\in\mathrm{SL}_2(\mathbb{Z})\},
\end{equation*}
then we can also determine all families $\{h_\mathbf{v}(\tau)\}_{\mathbf{v}\in\mathcal{V}_N}$
of functions in $\mathcal{F}_N$
satisfying (F1$^\prime$), (F2) and (F3) for arbitrary level $N\geq1$ (Theorem \ref{weaken}
and Remark \ref{weakenremark}).

\section {Galois actions on functions}

In this section we shall briefly describe the actions of the group $\mathrm{GL}_2(\mathbb{Z}/N\mathbb{Z})/\{\pm I_2\}\simeq
\mathrm{Gal}(\mathcal{F}_N/\mathcal{F}_1)$ on the field $\mathcal{F}_N$
based on \cite[$\S$6.1--6.2]{Shimura}.
\par
For a positive integer $N$, $\mathrm{GL}_2(\mathbb{Z}/N\mathbb{Z})/\{\pm I_2\}$ has a unique decomposition
\begin{equation*}
G_N\cdot
\mathrm{SL}_2(\mathbb{Z}/N\mathbb{Z})/\{\pm I_2\}\quad\textrm{with}\quad
G_N=\left\{\left[\begin{matrix}1&0\\0&d\end{matrix}\right]~|~
d\in(\mathbb{Z}/N\mathbb{Z})^\times\right\}.
\end{equation*}
This group acts on the field $\mathcal{F}_N$ as follows (\cite[$\S$6.1--6.2]{Shimura}): Let
\begin{equation*}
h(\tau)=\sum_{n\gg-\infty}c_nq^{n/N}\in\mathcal{F}_N\quad(c_n\in\mathbb{Q}(\zeta_N),~q=e^{2\pi i\tau}).
\end{equation*}
\begin{itemize}
\item[(A1)] The matrix $\left[\begin{matrix}1&0\\0&d\end{matrix}\right]\in G_N$
acts on $h(\tau)$ as
\begin{equation*}
h(\tau)^{\left[\begin{smallmatrix}1&0\\0&d\end{smallmatrix}\right]}=
\sum_{n\gg-\infty}c_n^{\sigma_d}q^{n/N},
\end{equation*}
where $\sigma_d$ is the automorphism of $\mathbb{Q}(\zeta_N)$
given by $\zeta_N^{\sigma_d}=\zeta_N^d$.
\item[(A2)] The matrix $\gamma\in\mathrm{SL}_2(\mathbb{Z}/N\mathbb{Z})/\{\pm I_2\}$
acts on $h(\tau)$ as
\begin{equation*}
h(\tau)^\gamma=(h\circ\widetilde{\gamma})(\tau),
\end{equation*}
where $\widetilde{\gamma}$ is any preimage of the reduction
$\mathrm{SL}_2(\mathbb{Z})\rightarrow\mathrm{SL}_2(\mathbb{Z}/N\mathbb{Z})/\{\pm I_2\}$.
\end{itemize}

\begin{lemma}\label{transitive}
Let $\{h_\mathbf{v}(\tau)\}_{\mathbf{v}\in\mathcal{V}_N}$ be a Fricke family of level $N\geq2$.
Then, $\mathrm{GL}_2(\mathbb{Z}/N\mathbb{Z})/\{\pm I_2\}$ acts
on $\{h_\mathbf{v}(\tau)\}_\mathbf{v}$ transitively.
\end{lemma}
\begin{proof}
Note by (F3) that $\mathrm{GL}_2(\mathbb{Z}/N\mathbb{Z})/\{\pm I_2\}$ acts
on the family $\{h_\mathbf{v}(\tau)\}_\mathbf{v}$.
Let $\mathbf{v}=\left[\begin{matrix}a/N\\b/N\end{matrix}\right]\in\mathcal{V}_N$, so $\gcd(a,b)$ is relatively prime to $N$.
If we take any
$\gamma=\left[\begin{matrix}a&b\\c&d\end{matrix}\right]\in
M_2(\mathbb{Z})$ such that $\det(\gamma)$ is relatively prime to $N$,
 then we find by (F3) that
\begin{equation*}
h_{\left[\begin{smallmatrix}1/N\\0\end{smallmatrix}\right]}(\tau)^\gamma=
h_{{^t}\gamma\left[\begin{smallmatrix}1/N\\0\end{smallmatrix}\right]}(\tau)=
h_{\left[\begin{smallmatrix}a/N\\b/N\end{smallmatrix}\right]}(\tau)=
h_\mathbf{v}(\tau).
\end{equation*}
This implies that $\mathrm{GL}_2(\mathbb{Z}/N\mathbb{Z})/\{\pm I_2\}$ acts
on $\{h_\mathbf{v}(\tau)\}_\mathbf{v}$ transitively.
\end{proof}

\begin{remark}
Roughly speaking, $\{h_\mathbf{v}(\tau)\}_\mathbf{v}$ is completely determined by its component
$h_{\left[\begin{smallmatrix}1/N\\0\end{smallmatrix}\right]}(\tau)$.
\end{remark}

\section {Fricke and Siegel functions}\label{sectwo}

For a lattice $\Lambda$ in $\mathbb{C}$, we let
\begin{equation*}
g_2(\Lambda)=60\sum_{\lambda\in\Lambda\setminus\{0\}}\frac{1}{\lambda^4},\quad
g_3(\Lambda)=140\sum_{\lambda\in\Lambda\setminus\{0\}}\frac{1}{\lambda^6}
\quad\textrm{and}\quad
\Delta(\Lambda)=g_2(\Lambda)^3-27g_3(\Lambda)^2.
\end{equation*}
The \textit{elliptic modular function} $j(\tau)$ is defined by
\begin{equation}\label{j}
j(\tau)=1728\frac{g_2(\tau)^3}{\Delta(\tau)}=1728\left(1+27\frac{g_3(\tau)^2}{\Delta(\tau)}\right)
\quad(\tau\in\mathbb{H}),
\end{equation}
where
$g_2(\tau)=g_2([\tau,1])$,
$g_3(\tau)=g_3([\tau,1])$ and
$\Delta(\tau)=\Delta([\tau,1])$.
This generates the ring of weakly holomorphic functions in $\mathcal{F}_1$
over $\mathbb{Q}$ (\cite[Theorem 2 in Chapter 5]{Lang}).
\par
The \textit{Weierstrass $\wp$-function} relative to $\Lambda$ is given by
\begin{equation*}
\wp(z;\Lambda)=
\frac{1}{z^2}+\sum_{\lambda\in\Lambda\setminus\{0\}}\left(\frac{1}{(z-\lambda)^2}-\frac{1}{\lambda^2}
\right)
\quad(z\in\mathbb{C}).
\end{equation*}
For each $\mathbf{v}=\left[\begin{matrix}v_1\\v_2
\end{matrix}\right]\in\mathbb{Q}^2\setminus\mathbb{Z}^2$, we define
the \textit{Fricke function}
$f_\mathbf{v}(\tau)$ by
\begin{equation}\label{Fricke}
f_\mathbf{v}(\tau)
=-2^73^5\frac{g_2(\tau)g_3(\tau)}{\Delta(\tau)}\wp_\mathbf{v}(\tau)
\quad(\tau\in\mathbb{H}),
\end{equation}
where $\wp_\mathbf{v}(\tau)=\wp(v_1\tau+v_2;[\tau,1])$.
\par
Furthermore, the \textit{Weierstrass $\sigma$-function} relative to $\Lambda$ is given by
\begin{equation*}
\sigma(z;\Lambda)=z\prod_{\lambda\in \Lambda\setminus\{0\}}\left(1-\frac{z}{\lambda}\right)
e^{z/\lambda+(1/2)(z/\lambda)^2}\quad(z\in\mathbb{C}).
\end{equation*}
Taking logarithmic derivative yields the \textit{Weierstrass
$\zeta$-function} as
\begin{equation*}
\zeta(z;\Lambda)=\frac{\sigma'(z;\Lambda)}{\sigma(z;\Lambda)}
=\frac{1}{z}+\sum_{\lambda\in
\Lambda\setminus\{0\}}\left(\frac{1}{z-\lambda}+\frac{1}{\lambda}+
\frac{z}{\lambda^2}\right)\quad(z\in\mathbb{C}).
\end{equation*}
Since $\zeta'(z;\Lambda)=-\wp(z;\Lambda)$ is periodic with respect to $\Lambda$, for each $\lambda\in\Lambda$ there is
a constant $\eta(\lambda;\Lambda)$ so that
\begin{equation*}
\zeta(z+\lambda;\Lambda)-\zeta(z;\Lambda)=\eta(\lambda;\Lambda)\quad(z\in\mathbb{C}).
\end{equation*}
For each $\mathbf{v}=\left[\begin{matrix}v_1\\v_2
\end{matrix}\right]\in\mathbb{Q}^2\setminus\mathbb{Z}^2$,
we then define
the \textit{Siegel function} $g_\mathbf{v}(\tau)$ by
\begin{equation}\label{Siegel}
g_\mathbf{v}(\tau)=
e^{-(v_1\eta(\tau;[\tau,1])+
v_2\eta(1;[\tau,1]))
(v_1\tau+v_2)/2}\sigma(v_1\tau+v_2;[\tau,1])\eta(\tau)^2
\quad(\tau\in\mathbb{H}),
\end{equation}
where
\begin{equation*}
\eta(\tau)=\sqrt{2\pi}\zeta_8q^{1/24}\prod_{n=1}^\infty
(1-q^n)\quad(\tau\in\mathbb{H})
\end{equation*}
is the \textit{Dedekind $\eta$-function} which is a $24$th root of $\Delta(\tau)$
(\cite[Theorem 5 in Chapter 18]{Lang}).
By the $q$-product expansion of the
Weierstrass $\sigma$-function
we get the expression
\begin{equation*}
g_{\mathbf{v}}(\tau)
=-e^{\pi i v_2(v_1-1)}
q^{(1/2)\mathbf{B}_2(v_1)}
(1-q^{v_1}e^{2\pi iv_2})
\prod_{n=1}^\infty (1-q^{n+v_1}e^{2\pi iv_2})
(1-q^{n-v_1}e^{-2\pi iv_2}),
\end{equation*}
where $\mathbf{B}_2(x)=x^2-x+1/6$
is the second Bernoulli polynomial
(\cite[Chapter 19, $\S$2]{Lang}). Note that
$g_\mathbf{v}(\tau)$ has neither zeros nor poles on $\mathbb{H}$.

\begin{proposition}\label{typical}
If $N\geq2$, then $\{f_\mathbf{v}(\tau)\}_{\mathbf{v}\in\mathcal{V}_N}$ and $\{g_\mathbf{v}(\tau)^{12N}\}_{\mathbf{v}\in\mathcal{V}_N}$ are Fricke families of level $N$.
\end{proposition}
\begin{proof}
See \cite[Chapter 6, $\S$2--3]{Lang} and \cite[Proposition 1.3 in Chapter 2]{K-L}.
\end{proof}

\begin{remark}
We call a function $h(\tau)$ in $\mathcal{F}_N$ a \textit{modular unit} of level $N\geq1$
if both $h(\tau)$ and $h(\tau)^{-1}$ are integral over $\mathbb{Q}[j(\tau)]$.
As is well known, $h(\tau)$ is a modular unit
if and only if it has neither zeros nor poles on $\mathbb{H}$ (\cite[p. 36]{K-L} or \cite[Proposition 2.3]{E-K-S}).
Thus $g_\mathbf{v}(\tau)^{12N}$ is a modular unit of level $N$ for every $\mathbf{v}\in\mathcal{V}_N$ with $N\geq2$. Moreover, $g_\mathbf{v}(\tau)$ is
a modular unit of level $12N^2$ (\cite[Theorems 5.2 and 5.3 in Chapter 3]{K-L}).
\end{remark}

For later use, we need the following lemma.

\begin{lemma}\label{difference}
Let $\mathbf{u},\mathbf{v}\in\mathbb{Q}^2\setminus\mathbb{Z}^2$.
\begin{itemize}
\item[\textup{(i)}] We have the assertion that $f_\mathbf{u}(\tau)=f_\mathbf{v}(\tau)$
if and only if $\mathbf{u}\equiv\pm\mathbf{v}\Mod{\mathbb{Z}^2}$.
\item[\textup{(ii)}] If $\mathbf{u}\not\equiv\pm\mathbf{v}\Mod{\mathbb{Z}^2}$, then we get the relation
\begin{equation*}
(f_\mathbf{u}(\tau)-f_\mathbf{v}(\tau))^6=
2^{12}3^6j(\tau)^2(j(\tau)-1728)^3
\frac{g_{\mathbf{u}+\mathbf{v}}(\tau)^6
g_{\mathbf{u}-\mathbf{v}}(\tau)^6}{g_\mathbf{u}(\tau)^{12}g_\mathbf{v}(\tau)^{12}}.
\end{equation*}
\item[\textup{(iii)}]
For each $\gamma\in\mathrm{SL}_2(\mathbb{Z})$, we get
$(g_\mathbf{v}\circ\gamma)(\tau)=\zeta g_{{^t}\gamma\mathbf{v}}(\tau)$ for some
$12$th root of unity $\zeta$ depending only on $\gamma$.
\end{itemize}
\end{lemma}
\begin{proof}
(i) See \cite[Lemma 10.4]{Cox} and definition (\ref{Fricke}).\\
(ii) See \cite[Theorem 2 in Chapter 18]{Lang} and definitions (\ref{j}), (\ref{Fricke}) and (\ref{Siegel}).\\
(iii) See \cite[Proposition 2.4]{K-S}.
\end{proof}

\section {Rings of weakly holomorphic functions}

For an integer $N\geq2$, we denote by $\mathrm{Fr}_N$ the set of all Fricke families of level $N$.
Then, $\mathrm{Fr}_N$ becomes a ring under the operations
\begin{equation}\label{operations}
\begin{array}{ccc}
\{h_\mathbf{v}(\tau)\}_\mathbf{v}+
\{k_\mathbf{v}(\tau)\}_\mathbf{v}
&=&\{(h_\mathbf{v}+k_\mathbf{v})(\tau)\}_\mathbf{v},\\
\{h_\mathbf{v}(\tau)\}_\mathbf{v}\cdot
\{k_\mathbf{v}(\tau)\}_\mathbf{v}
&=&\{(h_\mathbf{v}k_\mathbf{v})(\tau)\}_\mathbf{v}.
\end{array}
\end{equation}
\par
For a positive integer $N$, let $\mathcal{F}^1_N(\mathbb{Q})$ be the field of
meromorphic modular functions for
the congruence subgroup
\begin{equation*}
\Gamma^1(N)=\left\{\gamma\in\mathrm{SL}_2(\mathbb{Z})~|~
\gamma\equiv\left[\begin{matrix}1&0\\
\mathrm{*}&1\end{matrix}\right]
\Mod{N}\right\}
\end{equation*}
with rational Fourier coefficients.
Furthermore, we let $\mathcal{O}^1_N(\mathbb{Q})$ be its subring
consisting of weakly holomorphic functions.

\begin{lemma}\label{special}
Let $\{h_\mathbf{v}(\tau)\}_\mathbf{v}\in\mathrm{Fr}_N$ with $N\geq2$.
Then, $h_{\left[\begin{smallmatrix}1/N\\0\end{smallmatrix}\right]}(\tau)$ belongs to $\mathcal{O}^1_N(\mathbb{Q})$.
\end{lemma}
\begin{proof}
For any $\gamma\in\left[\begin{matrix}a&b\\c&d\end{matrix}\right]\in\Gamma^1(N)$,
we find that
\begin{eqnarray*}
h_{\left[\begin{smallmatrix}1/N\\0\end{smallmatrix}\right]}(\tau)^\gamma&=&
h_{{^t}\gamma\left[\begin{smallmatrix}1/N\\0\end{smallmatrix}\right]}(\tau)
\quad\textrm{by (F3)}\\
&=&
h_{\left[\begin{smallmatrix}a/N\\b/N\end{smallmatrix}\right]}(\tau)\\
&=&h_{\left[\begin{smallmatrix}1/N\\0\end{smallmatrix}\right]}(\tau)\quad
\textrm{by the fact $a\equiv1,~b\equiv0\Mod{N}$ and (F2)}.
\end{eqnarray*}
Thus $h_{\left[\begin{smallmatrix}1/N\\0\end{smallmatrix}\right]}(\tau)$ is modular for $\Gamma^1(N)$.
\par
Now, let $\alpha=\left[\begin{matrix}1&0\\0&d\end{matrix}\right]\in G_N$. We get
by (F3) and (F2) that
\begin{equation*}
h_{\left[\begin{smallmatrix}1/N\\0\end{smallmatrix}\right]}(\tau)^\alpha=
h_{{^t}\alpha\left[\begin{smallmatrix}1/N\\0\end{smallmatrix}\right]}(\tau)
=h_{\left[\begin{smallmatrix}1/N\\0\end{smallmatrix}\right]}(\tau),
\end{equation*}
which shows that $h_{\left[\begin{smallmatrix}1/N\\0\end{smallmatrix}\right]}(\tau)$ has
rational Fourier coefficients by (A1).
\par
Moreover, since $h_{\left[\begin{smallmatrix}1/N\\0\end{smallmatrix}\right]}(\tau)$ is weakly holomorphic by (F1),
it belongs to $\mathcal{O}^1_N(\mathbb{Q})$.
\end{proof}

Thus we obtain by Lemma \ref{special} a ring homomorphism
\begin{equation}\label{phiN}
\begin{array}{rccc}
\phi_N:&\mathrm{Fr}_N&\rightarrow&\mathcal{O}^1_N(\mathbb{Q})\\
&\{h_\mathbf{v}(\tau)\}_\mathbf{v}&\mapsto&h_{\left[\begin{smallmatrix}1/N\\0\end{smallmatrix}\right]}(\tau).
\end{array}
\end{equation}

\begin{lemma}\label{gamma}
Let $N\geq2$, and let $a$ and $b$ be a pair of integers such that $\gcd(a,b)$ is relatively prime to $N$. Let $\beta=\left[\begin{matrix}a&b\\c&d\end{matrix}\right]$
and $\beta'=\left[\begin{matrix}a&b\\c'&d'\end{matrix}\right]$ be matrices in $M_2(\mathbb{Z})$
such that $\det(\beta)\equiv\det(\beta')\equiv1\Mod{N}$. Then, there is a matrix $\alpha\in\Gamma^1(N)$ so that $\alpha\beta\equiv\beta'\Mod{N}$.
\end{lemma}
\begin{proof}
Take
$\alpha=\left[\begin{matrix}1&0\\c'd-cd'&1\end{matrix}\right]\in\Gamma^1(N)$. One can readily show that
\begin{equation*}
\alpha\beta\equiv\left[\begin{matrix}a&b\\c'\det(\gamma)+c(-\det(\gamma')+1) &
d'\det(\gamma)+d(-\det(\gamma')+1)\end{matrix}\right]\equiv\beta'\Mod{N}
\end{equation*}
due to the fact $\det(\beta)\equiv\det(\beta')\equiv1\Mod{N}$.
\end{proof}

\begin{theorem}\label{isomorphic}
If $N\geq2$, then
$\mathrm{Fr}_N$ and $\mathcal{O}^1_N(\mathbb{Q})$ are isomorphic via the map $\phi_N$ given in \textup{(\ref{phiN})}.
\end{theorem}
\begin{proof}
Let $\{h_\mathbf{v}(\tau)\}_\mathbf{v}\in\mathrm{ker}(\phi)$,
so $\phi_N(\{h_\mathbf{v}(\tau)\}_\mathbf{v})=h_{\left[\begin{smallmatrix}1/N\\0\end{smallmatrix}\right]}(\tau)=0$. Then we get by Lemma \ref{transitive} that $h_\mathbf{v}(\tau)=0$ for all $\mathbf{v}\in\mathcal{V}_N$. This shows that $\phi_N$ is one-to-one.
\par
Now, let $h(\tau)\in\mathcal{O}^1_N(\mathbb{Q})$. For each
$\mathbf{v}=\left[\begin{matrix}a/N\\b/N\end{matrix}\right]\in\mathcal{V}_N$, we take
any $\beta=\left[\begin{matrix}a&b\\c&d\end{matrix}\right]\in M_2(\mathbb{Z})$ such that
$\det(\beta)\equiv1\Mod{N}$, and set $h_\mathbf{v}(\tau)=h(\tau)^\beta$. We claim that $h_\mathbf{v}(\tau)$ is well-defined, independent of the choice of $\beta$.
Indeed, if $\beta'=\left[\begin{matrix}a&b\\c'&d'\end{matrix}\right]$ is another matrix in $M_2(\mathbb{Z})$ such that $\det(\beta')\equiv1\Mod{N}$, then we see that
\begin{eqnarray*}
h(\tau)^{\beta'}&=&h(\tau)^{\alpha\beta}\quad\textrm{for some}~\alpha\in\Gamma^1(N)~
\textrm{by Lemma \ref{gamma} and (\ref{Gal})}\\
&=&h(\tau)^\beta\quad\textrm{since $h(\tau)$ is modular for $\Gamma^1(N)$}.
\end{eqnarray*}
Since $h(\tau)$ is weakly holomorphic, so is $h_\mathbf{v}(\tau)=h(\tau)^\beta$ by (A2).
Furthermore, $h_\mathbf{v}(\tau)$ depends only on $\pm\mathbf{v}\Mod{\mathbb{Z}^2}$ by (\ref{Gal}).
Let $\gamma=\left[\begin{matrix}x&y\\z&w\end{matrix}\right]\in\mathrm{GL}_2(\mathbb{Z}/N\mathbb{Z})/\{\pm I_2\}$. We derive by considering $\gamma=\left[\begin{matrix}a&b\\c&d\end{matrix}\right]$
as an element of $\mathrm{SL}_2(\mathbb{Z}/N\mathbb{Z})/\{\pm I_2\}$ that
\begin{eqnarray*}
h_\mathbf{v}(\tau)^\gamma&=&
\left(h(\tau)^{\left[\begin{smallmatrix}a&b\\c&d\end{smallmatrix}\right]}\right)
^{\left[\begin{smallmatrix}x&y\\z&w\end{smallmatrix}\right]}\\
&=&h(\tau)^{\left[\begin{smallmatrix}ax+bz&ay+bw\\cx+dz&cy+dw\end{smallmatrix}\right]}\\
&=&\left(h(\tau)^{\left[\begin{smallmatrix}1&0\\0&\det(\gamma)\end{smallmatrix}\right]}\right)^
{\left[\begin{smallmatrix}ax+bz&ay+bw\\\det(\gamma)^{-1}(cx+dz)&\det(\gamma)^{-1}(cy+dw)\end{smallmatrix}\right]}\\
&=&h(\tau)^{\left[\begin{smallmatrix}ax+bz&ay+bw\\\det(\gamma)^{-1}(cx+dz)&\det(\gamma)^{-1}(cy+dw)\end{smallmatrix}\right]}
\quad\textrm{since $h(\tau)$ has rational Fourier coefficients}\\
&=&h_{\left[\begin{smallmatrix}(ax+bz)/N\\(ay+bw)/N\end{smallmatrix}\right]}(\tau)
\quad\textrm{because}~\left[\begin{smallmatrix}ax+bz&ay+bw\\\det(\gamma)^{-1}(cx+dz)&
\det(\gamma)^{-1}(cy+dw)\end{smallmatrix}\right]
\in\mathrm{SL}_2(\mathbb{Z}/N\mathbb{Z})/\{\pm I_2\}\\
&=&h_{\left[\begin{smallmatrix}x&z\\y&w\end{smallmatrix}\right]\left[\begin{smallmatrix}a/N\\b/N\end{smallmatrix}\right]}(\tau)\\
&=&h_{{^t}\gamma\mathbf{v}}(\tau).
\end{eqnarray*}
Thus the family $\{h_\mathbf{v}(\tau)\}_\mathbf{v}$ satisfies (F3).
Lastly, since
\begin{equation*}
\phi_N(\{h_\mathbf{v}(\tau)\}_\mathbf{v})=h_{\left[\begin{smallmatrix}1/N\\0\end{smallmatrix}\right]}(\tau),
\end{equation*}
$\phi_N$ is surjective.
\par
Therefore, we conclude that $\mathrm{Fr}_N$ and $\mathcal{O}^1_N(\mathbb{Q})$
are isomorphic via the map $\phi_N$.
\end{proof}

\section {Conjugate subgroups of $\mathrm{SL}_2(\mathbb{R})$}

For a positive integer $N$, let
\begin{equation*}
\Gamma_1(N)=\left\{\gamma\in\mathrm{SL}_2(\mathbb{Z})~|~
\gamma\equiv\left[\begin{matrix}1&\mathrm{*}\\0&1\end{matrix}\right]
\Mod{N}\right\}\quad\textrm{and}\quad
\omega_N=\left[\begin{matrix}1/\sqrt{N}&0\\0&\sqrt{N}\end{matrix}\right].
\end{equation*}
From the observation
\begin{equation*}
\omega_N\left[\begin{matrix}a&b\\c&d\end{matrix}\right]\omega_N^{-1}=
\left[\begin{matrix}a&b/N\\Nc&d\end{matrix}\right]
~\textrm{for all}~
\left[\begin{matrix}a&b\\c&d\end{matrix}\right]\in\mathrm{SL}_2(\mathbb{R}),
\end{equation*}
we note that $\Gamma^1(N)$ and $\Gamma_1(N)$ are conjugate in $\mathrm{SL}_2(\mathbb{R})$, namely
\begin{equation}\label{grouprelation}
\omega_N\Gamma^1(N)\omega_N^{-1}=\Gamma_1(N).
\end{equation}
Let $\mathcal{F}_{1,N}(\mathbb{Q})$ be the field of meromorphic modular functions
for $X_1(N)$ with rational Fourier coefficients. One can readily check that
the relation (\ref{grouprelation}) gives rise to an isomorphism
\begin{equation}\label{map}
\begin{array}{rcl}
\mathcal{F}_{1,N}(\mathbb{Q})&\stackrel{\sim}{\rightarrow}&\mathcal{F}^1_N(\mathbb{Q})\\
h(\tau)=\sum_{n\gg-\infty}c_nq^n&\mapsto&(h\circ\omega_N)(\tau)=h(\tau/N)
=\sum_{n\gg-\infty}c_nq^{n/N}
\end{array}
\end{equation}
with inverse map $f(\tau)\mapsto (f\circ\omega_N^{-1})(\tau)=f(N\tau)$.
Furthermore, we let
$\mathcal{O}_{1,N}(\mathbb{Q})$ be
the subring of $\mathcal{F}_{1,N}(\mathbb{Q})$ consisting of weakly holomorphic functions.
Since the map in (\ref{map}) preserves weakly holomorphicity,
it induces an isomorphism
\begin{equation}\label{induced}
\mathcal{O}_{1,N}(\mathbb{Q})\stackrel{\sim}{\rightarrow}
\mathcal{O}^1_N(\mathbb{Q}).
\end{equation}
\par
Let $X_1(4)$ be the modular curve corresponding to the group $\Gamma_1(4)$.
It is well known that
$X_1(4)$ is of genus $0$ with three inequivalent cusps
$0$, $1/2$ and $i\infty$ (\cite[p. 131]{K-K}). Moreover, the function
\begin{equation*}
g_{1,4}(\tau)=\left(\frac{g_{\left[\begin{smallmatrix}1/2\\0\end{smallmatrix}\right]}(4\tau)}
{g_{\left[\begin{smallmatrix}1/4\\0\end{smallmatrix}\right]}(4\tau)}\right)^{8}
=q^{-1}(1+q)^8\prod_{n=1}^\infty\left(\frac{(1-q^{4n+2})(1-q^{4n-2})}
{(1-q^{4n+1})(1-q^{4n-1})}\right)^8
\end{equation*}
generates the function field $\mathbb{C}(X_1(4))$ of $X_1(4)$
over $\mathbb{C}$, having values $16$, $0$ and $\infty$ at cusps
$0$, $1/2$ and $i\infty$, respectively
(\cite[Theorem 3 (ii)]{K-K} and \cite[Tables 2 and 3]{K-S}).
Since $g_{1,4}(\tau)$ has rational Fourier coefficients, we obtain
by \cite[Lemma 4.1]{K-K} that
\begin{equation}\label{F14}
\mathcal{F}_{1,4}(\mathbb{Q})=\mathbb{Q}(g_{1,4}(\tau)).
\end{equation}

\begin{lemma}\label{unitlevel4}
Let $c\in\mathbb{C}$. Then, $g_{1,4}(\tau)-c$ has neither zeros nor poles on $\mathbb{H}$
if and only if $c\in\{0,16\}$.
\end{lemma}
\begin{proof}
See \cite[(4)]{E-K-S}.
\end{proof}

\begin{theorem}\label{4structure}
We get the following structures.
\begin{itemize}
\item[\textup{(i)}] $\mathcal{O}_{1,4}(\mathbb{Q})=\mathbb{Q}[g_{1,4}(\tau),g_{1,4}(\tau)^{-1},
(g_{1,4}(\tau)-16)^{-1}]$.
\item[\textup{(ii)}] $\mathcal{O}^1_4(\mathbb{Q})=\mathbb{Q}[g^1_4(\tau),g^1_4(\tau)^{-1},
(g^1_4(\tau)-16)^{-1}]$, where
$g^1_4(\tau)=g_{1,4}(\tau/4)=g_{\left[\begin{smallmatrix}1/4\\0\end{smallmatrix}\right]}(\tau)^{-8}
g_{\left[\begin{smallmatrix}1/2\\0\end{smallmatrix}\right]}(\tau)^{8}$.
\end{itemize}
\end{theorem}
\begin{proof}
(i) Since $g_{1,4}(\tau)$ and $(g_{1,4}(\tau)-16)$ are modular units in $\mathcal{F}_{1,4}(\mathbb{Q})$ by Lemma \ref{unitlevel4} and (\ref{F14}),
we get the inclusion
$\mathcal{O}_{1,4}(\mathbb{Q})\supseteq\mathbb{Q}[g_{1,4}(\tau),g_{1,4}(\tau)^{-1},
(g_{1,4}(\tau)-16)^{-1}]$.
\par
Conversely, let $h(\tau)\in\mathcal{O}_{1,4}(\mathbb{Q})$.
By (\ref{F14}) we can express $h(\tau)$ as $h(\tau)=A(g_{1,4}(\tau))/B(g_{1,4}(\tau))$ for some
polynomials $A(x),B(x)\in\mathbb{Q}[x]$ which are relatively prime.
Suppose that $B(x)$ has a zero $c\in\overline{\mathbb{Q}}$ not equal to $0$ or $16$.
We note by Lemma \ref{unitlevel4} that $g_{1,4}(\tau_0)-c=0$ for some $\tau_0\in\mathbb{H}$,
from which it follows that $B(g_{1,4}(\tau_0))=0$. But since $A(x)$ is not divisible by $(x-c)$
in $\overline{\mathbb{Q}}[x]$,
we get $A(g_{1,4}(\tau_0))\neq0$. This contradicts that $h(\tau)$ is weakly holomorphic.
Thus $B(x)$ has no zeros other than $0$ and $16$, which implies the converse inclusion $\mathcal{O}_{1,4}(\mathbb{Q})
\subseteq\mathbb{Q}[g_{1,4}(\tau),g_{1,4}(\tau)^{-1},(g_{1,4}(\tau)-16)^{-1}]$.\\
(ii) This follows immediately from (i) and the isomorphism given in (\ref{induced}).
\end{proof}

\section {Generators for $N\equiv0\Mod{4}$}

Now we shall present explicit generators of the ring $\mathcal{O}^1_N(\mathbb{Q})$
over $\mathbb{Q}$ when $N\equiv0\Mod{4}$. This
amounts to classifying all Fricke families of level $N$ by Theorem \ref{isomorphic}.

\begin{proposition}\label{F1N}
If $N\geq2$, then we get $\mathcal{F}^1_N(\mathbb{Q})=\mathcal{F}_1(f_{\left[\begin{smallmatrix}
1/N\\0\end{smallmatrix}\right]}(\tau))$.
\end{proposition}
\begin{proof}
We recall that $\mathcal{F}_N$ is a Galois extension of $\mathcal{F}_1$ with
\begin{equation*}
\mathrm{Gal}(\mathcal{F}_N/\mathcal{F}_1)\simeq\mathrm{GL}_2(\mathbb{Z}/N(\mathbb{Z})/\{\pm I_2\}
\simeq G_N\cdot\mathrm{SL}_2(\mathbb{Z}/N\mathbb{Z})/\{\pm I_2\}.
\end{equation*}
Note by (A1) and (A2) that $\mathcal{F}_N$ is a Galois extension of
$\mathcal{F}^1_N(\mathbb{Q})$ with
\begin{equation*}
\mathrm{Gal}(\mathcal{F}_N/\mathcal{F}^1_N(\mathbb{Q}))\simeq
G_N\cdot\left\{\gamma\in\mathrm{SL}_2(\mathbb{Z}/N\mathbb{Z})/\{\pm I_2\}~|~
\gamma\equiv\pm\left[\begin{matrix}1&0\\\textrm{*}&1\end{matrix}\right]\Mod{N}\right\}.
\end{equation*}
Let $F=\mathcal{F}_1(f_{\left[\begin{smallmatrix}
1/N\\0\end{smallmatrix}\right]}(\tau))$.
Since $\{f_\mathbf{v}(\tau)\}_{\mathbf{v}\in\mathcal{V}_N}\in\mathrm{Fr}_N$
by Proposition \ref{typical},
we get the inclusion $F\subseteq\mathcal{F}^1_N(\mathbb{Q})$ by Lemma \ref{special}.
Suppose that $\gamma=\alpha\beta$ with $\alpha\in G_N$ and $\beta=\left[\begin{matrix}a&b\\c&d\end{matrix}\right]\in\mathrm{SL}_2(\mathbb{Z}/N\mathbb{Z})/\{\pm I_2\}$ leaves
$f_{\left[\begin{smallmatrix}
1/N\\0\end{smallmatrix}\right]}(\tau)$ fixed. We then derive that
\begin{eqnarray*}
f_{\left[\begin{smallmatrix}
1/N\\0\end{smallmatrix}\right]}(\tau)&=&
(f_{\left[\begin{smallmatrix}
1/N\\0\end{smallmatrix}\right]}(\tau)^\alpha)^\beta\\
&=&f_{\left[\begin{smallmatrix}
1/N\\0\end{smallmatrix}\right]}(\tau)^\beta\quad\textrm{since $f_{\left[\begin{smallmatrix}
1/N\\0\end{smallmatrix}\right]}(\tau)$ has rational Fourier coefficients}\\
&=&f_{{^t}\beta\left[\begin{smallmatrix}
1/N\\0\end{smallmatrix}\right]}(\tau)\quad\textrm{by (F2) and (F3)}\\
&=&f_{\left[\begin{smallmatrix}
a/N\\b/N\end{smallmatrix}\right]}(\tau).
\end{eqnarray*}
Thus we get $b\equiv0\Mod{N}$ and $a\equiv d\equiv\pm 1\Mod{N}$
by Lemma \ref{difference} (i) and the fact $\beta\in\mathrm{SL}_2(\mathbb{Z}/N\mathbb{Z})/\{\pm I_2\}$. This argument implies $F\supseteq\mathcal{F}^1_N(\mathbb{Q})$
by Galois theory. Therefore, we conclude that
$F=\mathcal{F}_1(f_{\left[\begin{smallmatrix}1/N\\0\end{smallmatrix}\right]}(\tau))=\mathcal{F}^1_N(\mathbb{Q})$. \end{proof}

When $N\geq8$ and $N\equiv0\Mod{4}$, we consider a function
\begin{equation*}
f^1_N(\tau)=\frac{f_{\left[\begin{smallmatrix}1/N\\0\end{smallmatrix}\right]}(\tau)
-f_{\left[\begin{smallmatrix}1/2\\0\end{smallmatrix}\right]}(\tau)}
{f_{\left[\begin{smallmatrix}1/4\\0\end{smallmatrix}\right]}(\tau)
-f_{\left[\begin{smallmatrix}1/2\\0\end{smallmatrix}\right]}(\tau)}\quad(\tau\in\mathbb{H}).
\end{equation*}
It is a modular unit belonging to $\mathcal{O}^1_N(\mathbb{Q})$ by Proposition \ref{typical},
Lemmas \ref{difference} (i) and \ref{special}.

\begin{theorem}\label{integral}
If $N\geq8$ and $N\equiv0\Mod{4}$, then we have
\begin{equation*}
\mathcal{O}^1_N(\mathbb{Q})=\mathcal{O}^1_4(\mathbb{Q})[f^1_N(\tau)]
=\mathbb{Q}[g^1_4(\tau),g^1_4(\tau)^{-1},
(g^1_4(\tau)-16)^{-1},f^1_N(\tau)].
\end{equation*}
\end{theorem}
\begin{proof}
It is obvious that $\mathcal{O}^1_N(\mathbb{Q})\supseteq\mathcal{O}^1_4(\mathbb{Q})[f^1_N(\tau)]$.
\par
For the converse inclusion, let $h(\tau)\in\mathcal{O}^1_N(\mathbb{Q})$. Note
by Proposition \ref{F1N} and Lemma \ref{special} that
\begin{equation*}
\mathcal{F}^1_N(\mathbb{Q})=\mathcal{F}_1(f_{\left[\begin{smallmatrix}1/N\\0\end{smallmatrix}\right]}(\tau))
=\mathcal{F}^1_4(\mathbb{Q})(f^1_N(\tau)).
\end{equation*}
So we can express $h=h(\tau)$ as
\begin{equation}\label{eqn}
h=c_0+c_1f+\cdots+c_{d-1}f^{d-1}
\end{equation}
where $d=[\mathcal{F}^1_N(\mathbb{Q}):\mathcal{F}^1_4(\mathbb{Q})]$
and $c_0,c_1,\ldots,c_{d-1}\in\mathcal{F}^1_4(\mathbb{Q})$.
Multiplying both sides of (\ref{eqn}) by $1,f,\ldots,f^{d-1}$, respectively,
yields a linear system (with unknowns $c_0,c_1,\ldots,c_{d-1}$)
\begin{equation*}
\left[\begin{matrix}
1&f&\cdots&f^{d-1}\\
f&f^2&\cdots&f^d\\
\vdots&\vdots&\ddots&\vdots\\
f^{d-1}&f^d&\cdots&f^{2d-2}
\end{matrix}\right]\left[\begin{matrix}c_0\\c_1\\\vdots\\c_{d-1}\end{matrix}\right]
=\left[\begin{matrix}h\\fh\\\vdots\\f^{d-1}h\end{matrix}\right].
\end{equation*}
By taking the trace $\mathrm{Tr}=\mathrm{Tr}_{\mathcal{F}^1_N(\mathbb{Q})/\mathcal{F}^1_4(\mathbb{Q})}$
on both sides we obtain
\begin{equation*}
T\left[\begin{matrix}c_0\\c_1\\\vdots\\c_{d-1}\end{matrix}\right]
=\left[\begin{matrix}\mathrm{Tr}(h)\\\mathrm{Tr}(fh)\\\vdots\\\mathrm{Tr}(f^{d-1}h)\end{matrix}\right]
\quad\textrm{with}\quad T=\left[\begin{matrix}
\mathrm{Tr}(1)&\mathrm{Tr}(f)&\cdots&\mathrm{Tr}(f^{d-1})\\
\mathrm{Tr}(f)&\mathrm{Tr}(f^2)&\cdots&\mathrm{Tr}(f^d)\\
\vdots&\vdots&\ddots&\vdots\\
\mathrm{Tr}(f^{d-1})&\mathrm{Tr}(f^d)&\cdots&\mathrm{Tr}(f^{2d-2})
\end{matrix}\right].
\end{equation*}
Since every $\mathrm{Tr}(*)$, appeared
in the above expression, lies in $\mathcal{O}^1_4(\mathbb{Q})$, we get
\begin{equation}\label{det-1}
c_0,c_1,\ldots,c_{d-1}\in\det(T)^{-1}\mathcal{O}^1_4(\mathbb{Q}).
\end{equation}
If we let $f_1,f_2,\ldots,f_d$ be all the Galois conjugates of $f$ over $\mathcal{F}^1_4(\mathbb{Q})$, then we find that
\begin{eqnarray*}
\det(T)&=&\left|\begin{matrix}
\sum_{k=1}^d f_k^0 & \sum_{k=1}^d f_k^1 & \cdots & \sum_{k=1}^d f_k^{d-1}\\
\sum_{k=1}^d f_k^1 & \sum_{k=1}^d f_k^2 & \cdots & \sum_{k=1}^d f_k^{d}\\
\vdots & \vdots & \ddots & \vdots\\
\sum_{k=1}^d f_k^{d-1} & \sum_{k=1}^d f_k^d & \cdots & \sum_{k=1}^d f_k^{2d-2}\end{matrix}
\right|\\
&=&
\left|\begin{matrix}
f_1^0 & f_2^0 & \cdots & f_d^0\\
f_1^1 & f_2^1 & \cdots & f_d^1\\
\vdots & \vdots & \ddots & \vdots\\
f_1^{d-1} & f_2^{d-1} & \cdots & f_d^{d-1}
\end{matrix}\right|
\left|\begin{matrix}
f_1^0 & f_1^1 & \cdots & f_1^{d-1}\\
f_2^0 & f_2^1 & \cdots & f_2^{d-1}\\
\vdots & \vdots & \ddots & \vdots\\
f_d^0 & f_d^1 & \cdots & f_d^{d-1}
\end{matrix}\right|\\
&=&\prod_{1\leq m<n\leq d}(f_m-f_n)^2\quad
\textrm{by the Vandermonde determinant formula}.
\end{eqnarray*}
On the other hand, since $f_{\left[\begin{smallmatrix}1/2\\0\end{smallmatrix}\right]}(\tau)$
and $f_{\left[\begin{smallmatrix}1/4\\0\end{smallmatrix}\right]}(\tau)$ belong to $\mathcal{F}^1_4(\mathbb{Q})$,
each $(f_m-f_n)$ is of the form
\begin{equation*}
f_m-f_n=
\frac{f_{\left[\begin{smallmatrix}a/N\\b/N\end{smallmatrix}\right]}(\tau)
-f_{\left[\begin{smallmatrix}1/2\\0\end{smallmatrix}\right]}(\tau)}
{f_{\left[\begin{smallmatrix}1/4\\0\end{smallmatrix}\right]}(\tau)-
f_{\left[\begin{smallmatrix}1/2\\0\end{smallmatrix}\right]}(\tau)}
-\frac{f_{\left[\begin{smallmatrix}c/N\\d/N\end{smallmatrix}\right]}(\tau)
-f_{\left[\begin{smallmatrix}1/2\\0\end{smallmatrix}\right]}(\tau)}
{f_{\left[\begin{smallmatrix}1/4\\0\end{smallmatrix}\right]}(\tau)-
f_{\left[\begin{smallmatrix}1/2\\0\end{smallmatrix}\right]}(\tau)}
=\frac{f_{\left[\begin{smallmatrix}a/N\\b/N\end{smallmatrix}\right]}(\tau)
-f_{\left[\begin{smallmatrix}c/N\\d/N\end{smallmatrix}\right]}(\tau)}
{f_{\left[\begin{smallmatrix}1/4\\0\end{smallmatrix}\right]}(\tau)-
f_{\left[\begin{smallmatrix}1/2\\0\end{smallmatrix}\right]}(\tau)}
\end{equation*}
for some $\left[\begin{matrix}a/N\\b/N\end{matrix}\right],\left[\begin{matrix}c/N\\d/N\end{matrix}\right]\in
\mathcal{V}_N$ such that $\left[\begin{matrix}a/N\\b/N\end{matrix}\right]\not
\equiv\pm\left[\begin{matrix}c/N\\d/N\end{matrix}\right]\Mod{\mathbb{Z}^2}$.
Thus $\det(T)$ is a modular unit in $\mathcal{O}^1_4(\mathbb{Q})$ by Lemma \ref{difference} (i),
from which it follows by (\ref{eqn}) and (\ref{det-1}) that $h(\tau)\in\mathcal{O}^1_4(\mathbb{Q})[f^1_N(\tau)]$.
Therefore we achieve the inclusion $\mathcal{O}^1_N(\mathbb{Q})\subseteq\mathcal{O}^1_4(\mathbb{Q})[f^1_N(\tau)]$, as desired.
\end{proof}

\begin{corollary}\label{expression}
Let $N\geq8$ and $N\geq0\Mod{4}$.
For each $\mathbf{v}=\left[\begin{matrix}a/N\\b/N\end{matrix}\right]\in\mathcal{V}_N$, let
\begin{equation*}
r_\mathbf{v}(\tau)=\left(\frac{g_{(N/2)\mathbf{v}}(\tau)}{g_{(N/4)\mathbf{v}}(\tau)}\right)^8\quad\textrm{and}\quad
s_\mathbf{v}(\tau)=\frac{f_{\mathbf{v}}(\tau)
-f_{(N/2)\mathbf{v}}(\tau)}
{f_{(N/4)\mathbf{v}}(\tau)
-f_{(N/2)\mathbf{v}}(\tau)}.
\end{equation*}
Then, a family $\{h_\mathbf{v}(\tau)\}_{\mathbf{v}\in\mathcal{V}_N}$ of functions in $\mathcal{F}_N$
is a Fricke family of level $N$ if and only if
there is a polynomial $P(x,y,z,w)\in\mathbb{Q}[x,y,z,w]$ so that
\begin{equation*}
h_\mathbf{v}(\tau)=P(r_\mathbf{v}(\tau),r_\mathbf{v}(\tau)^{-1},(r_\mathbf{v}(\tau)-16)^{-1},
s_\mathbf{v}(\tau))\quad
\textrm{for all}~\mathbf{v}\in\mathcal{V}_N.
\end{equation*}
\end{corollary}
\begin{proof}
Take any $\gamma=\left[\begin{matrix}a&b\\c&d\end{matrix}\right]\in M_2(\mathbb{Z})$
such that $\det(\gamma)\equiv1\Mod{N}$.
We find by (A1), Lemma \ref{difference} (iii) and Proposition \ref{typical} that
\begin{equation*}
r_\mathbf{v}(\tau)=g^1_4(\tau)^{\gamma}\quad\textrm{and}
\quad s_\mathbf{v}(\tau)=f^1_N(\tau)^{\gamma}.
\end{equation*}
Now, the result follows from Theorems \ref{isomorphic} (with its proof) and \ref{integral}.
\end{proof}

\section {Weak Fricke families}

Let $\mathbb{H}^\prime=\mathbb{H}\setminus\{\gamma(\zeta_3),\gamma(\zeta_4)~|~\gamma\in\mathrm{SL}_2(\mathbb{Z})\}$.
For a positive integer $N$,
let ${\mathcal{O}^1_N}^\prime(\mathbb{Q})$ be the ring of functions in $\mathcal{F}^1_N(\mathbb{Q})$
which are holomorphic on $\mathbb{H}^\prime$.

\begin{lemma}\label{jbijection}
$j(\tau)$ gives to rise a bijection $j(\tau):\mathrm{SL}_2(\mathbb{Z})\backslash\mathbb{H}\rightarrow\mathbb{C}$ such that
$j(\zeta_3)=0$ and $j(\zeta_4)=1728$.
\end{lemma}
\begin{proof}
See \cite[Theorem 4 in Chapter 3]{Lang}.
\end{proof}

\begin{theorem}\label{1weaken}
We have ${\mathcal{O}^1_1}^\prime(\mathbb{Q})=\mathbb{Q}[j(\tau),j(\tau)^{-1},(j(\tau)-1728)^{-1}]$.
\end{theorem}
\begin{proof}
By Lemma \ref{jbijection} we get the inclusion
${\mathcal{O}^1_1}^\prime(\mathbb{Q})\supseteq\mathbb{Q}[j(\tau),j(\tau)^{-1},(j(\tau)-1728)^{-1}]$.
\par
Now, let $h(\tau)\in{\mathcal{O}^1_1}^\prime(\mathbb{Q})$. Since $\mathcal{F}^1_1(\mathbb{Q})=\mathcal{F}_1=\mathbb{Q}(j(\tau))$, we have
$h(\tau)=A(j(\tau))/B(j(\tau))$ for some polynomials
$A(x),B(x)\in\mathbb{Q}[x]$ which are relatively prime.
Suppose that $B(x)$ has a zero $c\in\overline{\mathbb{Q}}$ not equal to $0$ and $1728$.
Since $j(\tau_0)=c$ for some $\tau_0\in\mathbb{H}^\prime$ by Lemma \ref{jbijection}, we have
$B(j(\tau_0))=0$. But since $A(x)$ is not divisible by $(x-c)$,
we see that $A(j(\tau_0))\neq0$,
which contradicts that $h(\tau)$ is holomorphic on $\mathbb{H}^\prime$. Thus we conclude that $0$ and $1728$ are the only possible zeros of $B(x)$. This proves the converse inclusion ${\mathcal{O}^1_1}^\prime(\mathbb{Q})\subseteq\mathbb{Q}[j(\tau),j(\tau)^{-1},(j(\tau)-1728)^{-1}]$.
\end{proof}

\begin{lemma}\label{1unit}
Modular units of level $1$ are exactly nonzero rational numbers.
\end{lemma}
\begin{proof}
See \cite[Lemma 2.1]{K-S}. One can also give a proof by using Lemma \ref{jbijection}.
\end{proof}

\begin{theorem}\label{weaken}
If $N\geq2$, then we obtain
\begin{equation*}
{\mathcal{O}^1_N}^\prime(\mathbb{Q})={\mathcal{O}^1_1}^\prime(\mathbb{Q})[f_{\left[\begin{smallmatrix}1/N\\0\end{smallmatrix}\right]}(\tau)]
=\mathbb{Q}[j(\tau),j(\tau)^{-1},(j(\tau)-1728)^{-1},f_{\left[\begin{smallmatrix}1/N\\0\end{smallmatrix}\right]}(\tau)].
\end{equation*}
\end{theorem}
\begin{proof}
Since $f_{\left[\begin{smallmatrix}1/N\\0\end{smallmatrix}\right]}(\tau)$ is weakly holomorphic,
we get the inclusion ${\mathcal{O}^1_N}^\prime(\mathbb{Q})\supseteq{\mathcal{O}^1_1}^\prime(\mathbb{Q})[f_{\left[\begin{smallmatrix}1/N\\0\end{smallmatrix}\right]}(\tau)]$.
\par
For the converse inclusion, let $h(\tau)\in{\mathcal{O}^1_N}^\prime(\mathbb{Q})$. Since $\mathcal{F}^1_N(\mathbb{Q})$
is generated by $f=f_{\left[\begin{smallmatrix}
1/N\\0\end{smallmatrix}\right]}(\tau)$ over $\mathcal{F}_1=\mathcal{F}^1_1(\mathbb{Q})$
by Proposition \ref{F1N}, we can write
\begin{equation}\label{linear}
h=c_0+c_1f+\cdots+c_{d-1}f^{d-1}
\end{equation}
where $d=[\mathcal{F}^1_N(\mathbb{Q}):\mathcal{F}^1_1(\mathbb{Q})]$ and
$c_0,c_1,\ldots,c_{d-1}\in\mathcal{F}^1_1(\mathbb{Q})$. If $f_1,f_2,\ldots,f_d$ denotes
all the Galois conjugates of $f$ over $\mathcal{F}^1_1(\mathbb{Q})$,
and $D=\prod_{1\leq m,n\leq d}(f_m-f_n)^2$, then one can show that
\begin{equation}\label{D-1}
c_0,c_1,\ldots,c_{d-1}\in D^{-1}{\mathcal{O}^1_1}^\prime(\mathbb{Q}).
\end{equation}
as in the proof of Theorem \ref{integral}. By Lemma \ref{difference} (ii) we see that each
$(f_m-f_n)^6$ is of the form
\begin{equation*}
(f_m-f_n)^6=
2^{12}3^6j(\tau)^2(j(\tau)-1728)^3
\frac{g_{\mathbf{u}+\mathbf{v}}(\tau)^6
g_{\mathbf{u}-\mathbf{v}}(\tau)^6}{g_\mathbf{u}(\tau)^{12}g_\mathbf{v}(\tau)^{12}}
\end{equation*}
for some $\mathbf{u},\mathbf{v}\in\mathcal{V}_N$ such that $\mathbf{u}
\not\equiv\pm\mathbf{v}\Mod{\mathbb{Z}^2}$.
It follows by Lemma \ref{1unit} that
\begin{equation*}
D=cj(\tau)^{d(d-1)/3}(j(\tau)-1728)^{d(d-1)/2}\quad\textrm{for some nonzero}~c\in\mathbb{C}.
\end{equation*}
Since $D\in\mathcal{F}^1_1(\mathbb{Q})=\mathbb{Q}(j(\tau))$, we must have
$d(d-1)/3\in\mathbb{Z}$ and $c\in\mathbb{Q}$. Therefore we achieve by Theorem \ref{1weaken},
(\ref{linear}) and (\ref{D-1}) that
$h(\tau)\in{\mathcal{O}^1_1}^\prime(\mathbb{Q})[f_{\left[\begin{smallmatrix}1/N\\0\end{smallmatrix}\right]}(\tau)]$. Hence the inclusion
${\mathcal{O}^1_N}^\prime(\mathbb{Q})\subseteq{\mathcal{O}^1_1}^\prime(\mathbb{Q})[f_{\left[\begin{smallmatrix}1/N\\0\end{smallmatrix}\right]}(\tau)]$
also holds.
\end{proof}

\begin{remark}\label{weakenremark}
For $N\geq2$,
let $\mathrm{Fr}^\prime_N$ be the set of \textit{weak} Fricke families of level $N$, namely,
the families $\{h_\mathbf{v}(\tau)\}_{\mathbf{v}\in\mathcal{V}_N}$ of functions
in $\mathcal{F}_N$ satisfying (F1), (F2$^\prime$) and (F3).
It is also a ring under the operations as in (\ref{operations}).
In a similar way to the proof of Theorem \ref{isomorphic},
one can readily show that $\mathrm{Fr}^\prime_N$ is isomorphic to
${\mathcal{O}^1_N}^\prime(\mathbb{Q})$. Therefore, we deduce by Theorem \ref{weaken} that
a family $\{h_\mathbf{v}(\tau)\}_{\mathbf{v}\in\mathcal{V}_N}$ of functions in $\mathcal{F}_N$
is a weak Fricke family of level $N$ if and only if there is a polynomial $P(x,y,z,w)\in\mathbb{Q}[x,y,z,w]$ so that
\begin{equation*}
h_\mathbf{v}(\tau)=P(j(\tau),j(\tau)^{-1},(j(\tau)-1728)^{-1},f_{\left[\begin{smallmatrix}1/N\\0\end{smallmatrix}\right]}(\tau))
\quad\textrm{for all}~\mathbf{v}\in\mathcal{V}_N.
\end{equation*}
\end{remark}

\bibliographystyle{amsplain}

\address{
School of Mathematics \\
Korea Advanced Institute for Advanced Study\\
Seoul 02455\\
Republic of Korea} {zandc@kias.re.kr}
\address{
Department of Mathematics\\
Hankuk University of Foreign Studies\\
Yongin-si, Gyeonggi-do 17035\\
Republic of Korea} {dhshin@hufs.ac.kr}

\end{document}